\documentclass[10pt,reqno]{amsart}

\usepackage{amssymb}
\usepackage{amsmath}
\usepackage{amsfonts}
\usepackage{amscd}
\usepackage{mathrsfs}

\newcommand{\B}{\mathcal{B}}
\newcommand{\C}{\mathbf{C}}
\newcommand{\E}{\mathcal{E}}
\newcommand{\Ic}{\mathtt{i}\hspace{0.2mm}}
\newcommand{\M}{\mathscr{M}}
\newcommand{\N}{\mathscr{H}}
\newcommand{\X}{\mathscr{X}}

\newtheorem{theorem}{Theorem}[section]
\newtheorem{corollary}[theorem]{Corollary}
\newtheorem{definition}[theorem]{Definition}

\newtheorem{proposition}[theorem]{Proposition}

\newtheorem{remark}[theorem]{Remark}
\numberwithin{equation}{section}

\newcommand{\br}[1]{\left[#1\right]}
\newcommand{\bre}[1]{\left\{#1\right\}}
\newcommand{\n}[1]{\left\vert#1\right\vert}

\newcommand{\pr}[1]{\left(#1\right)}

\begin{document}

\title[]{On the value distribution of the Riemann zeta-function and the Euler gamma-function}

\author[]{Qi Han$^{1,\dag}$, Jingbo Liu$^{1,2}$, and Qiong Wang$^{3,4}$}

\address{$^1$ Department of Mathematics, Texas A\&M University, San Antonio, Texas 78224, USA
\vskip 2pt $^2$ Department of Mathematics, Wesleyan University, Middletown, Connecticut 06459, USA
\vskip 2pt \hspace{1.4mm} Email: {\sf qhan@tamusa.edu (QH) \hspace{0.6mm} jliu@tamusa.edu (JL) \hspace{0.6mm} jliu02@wesleyan.edu (JL)}
\vskip 2pt $^3$ Department of Mathematics, Shandong University, Jinan, Shandong 250100, P.R. China
\vskip 2pt $^4$ Department of Mathematics, University of California, Irvine, California 92697, USA
\vskip 2pt \hspace{1.4mm} Email: {\sf qiongwangsdu@126.com (QW) \hspace{0.6mm} qiongw11@uci.edu (QW)}}

\thanks{$^\dag$Qi Han is the corresponding author of this research work.}


\begin{abstract}
  We prove some uniqueness results for the Riemann zeta-function and the Euler gamma-function by virtue of shared values using the value distribution theory.
\end{abstract}

\maketitle

\section{Introduction}\label{Int} 
Recall the Riemann zeta-function $\zeta$ is originally defined through the Dirichlet series
\begin{equation}
\zeta(s)=\sum_{n=1}^{+\infty}\frac{1}{n^s},\nonumber
\end{equation}
with $\Re(s)>1$, that can be analytically continued to be a meromorphic function in the complex variable $s=\sigma+\Ic t\in\C$ having only a simple pole at $s=1$.
The famous, yet unproven, Riemann hypothesis asserts that all the non-trivial zeros of $\zeta$ lie on the line $\Re(s)=\sigma=\frac{1}{2}$.

Value distribution of the Riemann zeta-function has been studied extensively.
See the classic by Titchmarsh \cite{Ti} and a recent monograph from Steuding \cite{St}; results more closely related to Nevanlinna theory can be found in Liao-Yang \cite{LY} and Ye \cite{Ye}.
It is well-known by Nevanlinna \cite{Ne1} that a non-constant meromorphic function $f$ in $\C$ is completely determined by ``$5$ IM'' value sharing condition (ignoring multiplicity), whereas $f$ is completely determined by ``$3$ CM'' value sharing condition (counting multiplicity) when $f$ is further assumed to be of finite non-integral order.
Han \cite{Ha} recently proved the mixed ``$1$ CM + $3$ IM'' value sharing condition sufficient if $f$ has finite non-integral order, or $f$ has integral order yet is of maximal growth type; this result particularly applies to the Euler gamma-function and the Riemann zeta-function.

In this paper, we discuss some uniqueness results primarily about the Riemann zeta-function and the Euler gamma-function in light of their nice properties.
Specifically, we will prove some results for the Riemann zeta-function that extends Gao-Li \cite{GL} in section 2, and will prove some results for the Euler gamma-function that extends Liao-Yang \cite{LY} in section 3; finally, in section 4, we will revisit the Riemann zeta-function and discuss some other related results.

Below, we assume the reader is familiar with the basics of Nevanlinna theory of meromorphic functions in $\C$ such as the first and second main theorems, and the standard notations such as the characteristic function $T(r,f)$, the proximity function $m(r,f)$, and the integrated counting functions $N(r,f)$ (counting multiplicity) and $\bar{N}(r,f)$ (ignoring multiplicity).
$S(r,f)$ denotes a quantity satisfying $\displaystyle{\displaystyle{S(r,f)=o(T(r,f))}}$ as $r\to+\infty$, since $f$ is (always assumed to be) of finite order.
Here, the order of $f$ is defined as $\displaystyle{\rho(f):=\limsup\limits_{r\to+\infty}\frac{\log T(r,f)}{\log r}}$.
Excellent references of this theory are certainly Nevanlinna \cite{Ne2}, Hayman \cite{Hay}, Yang \cite{Ya}, and Cherry-Ye \cite{CY} etc.

\section{Results regarding the Riemann zeta-function I}\label{zeta1} 
The first result that we shall need is essentially due to Levinson \cite {Le}, while for convenience of the reader we reproduce it as the following proposition.

\begin{proposition}\label{Pr1}
Given $a\in\C$, $\zeta(s)-a$ has infinitely many zeros in the strips
\begin{equation}
Z_V:=\bre{s:\frac{1}{4}<\sigma<\frac{3}{4},\,t>0}\hspace{2mm}and\hspace{2mm}Z_H:=\bre{s:-2<t<2,\,\sigma<0},\nonumber
\end{equation}
respectively, where $t\to+\infty$ in $Z_V$ and $\sigma\to-\infty$ in $Z_H$.
\end{proposition}

Now, we can formulate our main results of this section as follows.

\begin{theorem}\label{Th1}
Let $a,b,c\in\C$ be finite and distinct, and let $f$ be a meromorphic function in $\C$ such that either $\displaystyle{\limsup\limits_{r\to+\infty}\frac{\bar{N}\big(r,\frac{1}{f-c}\big)}{T(r,f)}<1}$ or $\displaystyle{\bar{N}\Big(r,\frac{1}{f-c}\Big)=O(T(r,\zeta))}$ with $\zeta$ being the Riemann zeta-function.
When $f,\zeta$ share the finite values $a,b$ counting multiplicity, with $\displaystyle{\zeta^{-1}(c)\subseteq f^{-1}(c)}$ ignoring multiplicity, except possibly at finitely many points, then $f=\zeta$.
\end{theorem}

\begin{proof}
Consider an auxiliary function
\begin{equation}\label{Eq-1}
F:=\frac{\zeta-a}{f-a}\cdot\frac{f-b}{\zeta-b}.
\end{equation}
Since $f,\zeta$ share $a,b$ CM except possibly at finitely many points, one knows $F$ is a meromorphic function having only finitely many zeros and poles.
Thus, there exists an entire function $g$ and a rational function $R$ such that
\begin{equation}\label{Eq-2}
F=\frac{\zeta-a}{f-a}\cdot\frac{f-b}{\zeta-b}=R\hspace{0.2mm}e^g.
\end{equation}

We next claim $g$ is linear.
In fact, from \cite[Lemma 2.7]{LY} or \cite[Theorem 1]{Ye}, one has
\begin{equation}\label{Eq-3}
T(r,\zeta)=\frac{1}{\pi}\,r\log r+O(r)
\end{equation}
so that $\rho(\zeta)=1$.
Using Nevanlinna's first and second main theorems, we observe that
\begin{equation}\label{Eq-4}
\begin{aligned}
T(r,f)&\leq\bar{N}\Big(r,\frac{1}{f-a}\Big)+\bar{N}\Big(r,\frac{1}{f-b}\Big)+\bar{N}\Big(r,\frac{1}{f-c}\Big)+S(r,f)\\
      &=\bar{N}\Big(r,\frac{1}{\zeta-a}\Big)+\bar{N}\Big(r,\frac{1}{\zeta-b}\Big)+\bar{N}\Big(r,\frac{1}{f-c}\Big)+O(\log r)+S(r,f)\\
      &\leq2\hspace{0.2mm}T(r,\zeta)+\bar{N}\Big(r,\frac{1}{f-c}\Big)+O(\log r)+S(r,f).
\end{aligned}
\end{equation}
Now, when $\displaystyle{\iota:=\limsup\limits_{r\to+\infty}\frac{\bar{N}\big(r,\frac{1}{f-c}\big)}{T(r,f)}<1}$ holds, one has from \eqref{Eq-4} that
\begin{equation}\label{Eq-5}
\pr{1-\frac{1+\iota}{2}-o(1)}T(r,f)\leq2\hspace{0.2mm}T(r,\zeta)+O(\log r);
\end{equation}
while when $\displaystyle{\bar{N}\Big(r,\frac{1}{f-c}\Big)=O(T(r,\zeta))}$ holds, one sees from \eqref{Eq-4} that
\begin{equation}\label{Eq-6}
\pr{1-o(1)}T(r,f)\leq O(T(r,\zeta))+O(\log r).
\end{equation}
Therefore, via \eqref{Eq-3}, one derives that $\displaystyle{T(r,f)\leq O(r\log r)+O(r)}$ (for all $r$ outside of a possible set of finite Lebesgue measure), which implies $\rho(f)\leq1$ (so there is no exceptional set).
Hence, it is routine to note $\rho(F)\leq1$ as well.
Through the renowned Hadamard factorization theorem (see Berenstein-Gay \cite[Section 4.6.15]{BG}), we know $g$ is linear as claimed earlier.

We further assert $g\equiv0$; otherwise, suppose that $g=As$ where $A\neq0$ is a complex number.
Recall $\displaystyle{\zeta^{-1}(c)\subseteq f^{-1}(c)}$ except possibly at finitely many points.
Via Proposition \ref{Pr1}, there are infinitely many zeros of $\zeta-c$ in both the strips $Z_V$ and $Z_H$ (since $c\neq\infty$).
Note these zeros of $\zeta-c$ are zeros of $R\hspace{0.2mm}e^{As}-1$.
Denote these zeros by $\varpi_n=\beta_n+\Ic\gamma_n$ in the vertical strip $Z_V$ and $\omega_n=\mu_n+\Ic\nu_n$ in the horizontal strip $Z_H$, respectively.
Obviously, $\gamma_n\to+\infty$ and $\mu_n\to-\infty$ when $n\to+\infty$, and $\beta_n$'s and $\nu_n$'s are uniformly bounded.
Write $A=B+\Ic C$ with $B,C$ real numbers.
Then, we get
\begin{equation}\label{Eq-7}
1\equiv R(\varpi_n)\hspace{0.2mm}e^{A\varpi_n}=R(\beta_n+\Ic\gamma_n)\hspace{0.2mm}e^{(B+\Ic C)(\beta_n+\Ic\gamma_n)},
\end{equation}
\begin{equation}\label{Eq-8}
1\equiv R(\omega_n)\hspace{0.2mm}e^{A\hspace{0.2mm}\omega_n}=R(\mu_n+\Ic\nu_n)\hspace{0.2mm}e^{(B+\Ic C)(\mu_n+\Ic\nu_n)}.
\end{equation}
We consider two different cases.\\
{\bf Case 1.} $B\neq0$.
Then, we apply \eqref{Eq-8} to deduce
$\displaystyle{1\equiv\n{R(\omega_n)}e^{B\mu_n-C\nu_n}\to\left\{\begin{array}{ll}
0 &\mathrm{if}\hspace{2mm}B>0\medskip\\
+\infty &\mathrm{if}\hspace{2mm}B<0
\end{array}\right.}$ with $\mu_n\to-\infty$ and bounded $\nu_n$'s, since $R(s)$ is rational and $e^{As}$ is exponential.\\
{\bf Case 2.} $B=0$.
If $C\neq0$, we then use \eqref{Eq-7} to show
$\displaystyle{1\equiv\n{R(\varpi_n)}e^{-C\gamma_n}\to\left\{\begin{array}{ll}
0 &\mathrm{if}\hspace{2mm}C>0\medskip\\
+\infty &\mathrm{if}\hspace{2mm}C<0
\end{array}\right.}$ with $\gamma_n\to+\infty$ and bounded $\beta_n$'s, since $R(s)$ is rational and $e^{As}$ is exponential.

As a result, $B=C=0$ so that the assertion $g\equiv0$ is proved, which leads to
\begin{equation}\label{Eq-9}
F=\frac{\zeta-a}{f-a}\cdot\frac{f-b}{\zeta-b}=R.
\end{equation}
As $\displaystyle{\zeta^{-1}(c)\subseteq f^{-1}(c)}$ except possibly at finitely many points, and $\zeta-c=0$ has infinitely many roots, we must have $R(s)\equiv1$ noticing $R$ is rational.
It then follows that $f=\zeta$.
\end{proof}

\begin{theorem}\label{Th2}
Let $a\neq c\in\C$ be finite, and let $f$ be a meromorphic function in $\C$ having only finitely many poles such that either $\displaystyle{\limsup\limits_{r\to+\infty}\frac{\bar{N}\big(r,\frac{1}{f-c}\big)}{T(r,f)}<1}$ or $\displaystyle{\bar{N}\Big(r,\frac{1}{f-c}\Big)=O(T(r,\zeta))}$ where $\zeta$ is the Riemann zeta-function.
When $f,\zeta$ share $a$ counting multiplicity, with $\displaystyle{\zeta^{-1}(c)\subseteq f^{-1}(c)}$ ignoring multiplicity, except possibly at finitely many points, then $f=\zeta$.
\end{theorem}

\begin{proof}
Recall $\zeta$ has its unique simple pole at $s=1$.
Consider an auxiliary function
\begin{equation}\label{Eq-10}
F_1:=\frac{\zeta-a}{f-a}.
\end{equation}
In view of the assumptions that $f,\zeta$ share $a$ CM except possibly at finitely many points and $f$ has only finitely many poles, $F_1$ is a meromorphic function with finitely many zeros and poles.
Hence, there is an entire function $g$ and a rational function $R$ such that
\begin{equation}\label{Eq-11}
F_1=\frac{\zeta-a}{f-a}=R\hspace{0.2mm}e^g.
\end{equation}

Next, one can exploit exactly the same analyses as described in the proof of Theorem \ref{Th1} to deduce that $g\equiv0$.
We then have
\begin{equation}\label{Eq-12}
F_1=\frac{\zeta-a}{f-a}=R.
\end{equation}
As $\displaystyle{\zeta^{-1}(c)\subseteq f^{-1}(c)}$ except possibly at finitely many points, and $\zeta-c=0$ has infinitely many roots, we must have $R(s)\equiv1$ noticing $R$ is rational.
It then follows that $f=\zeta$.
\end{proof}

\begin{remark}\label{Re1}
Li \cite{Li1} considered the uniqueness of an $L$-function in the extended Selberg class regarding a general meromorphic function in $\C$ having only finitely many poles, and proved the the combined {\rm``$1$ CM + $1$ IM''} value sharing condition sufficient.
Extensions of this nice result were given by Garunk\v{s}tis-Grahl-Steuding \cite{GGS} and Han \cite{Ha}.
On the other hand, when we focus on the Riemann zeta-function $\zeta$, results using 3 general value sharing condition were discussed by Gao-Li \cite{GL} which answered an open question of Liao-Yang \cite{LY}.
In \cite{GL}, the authors also provided an example to show that $\displaystyle{\zeta^{-1}(c)\subseteq f^{-1}(c)}$ ignoring multiplicity is not adequate in proving $f=\zeta$ besides the other hypothesis that $f,\zeta$ share $0,1$ {\rm CM}, with $\displaystyle{f=-\frac{\zeta}{\big(e^{(\zeta-c)(s-1)}-1\big)\zeta-e^{(\zeta-c)(s-1)}}}$ and $c\neq0,1$.
For this $f$, one has $\displaystyle{T(r,f)=\bar{N}\Big(r,\frac{1}{f-c}\Big)+O(1)}$ and $\displaystyle{\liminf\limits_{r\to+\infty}\frac{\bar{N}\big(r,\frac{1}{f-c}\big)}{T(r,\zeta)}=+\infty}$; this observation leads to the growth constraints about $f-c$ in Theorems \ref{Th1} and \ref{Th2}.
\end{remark}

\section{Results regarding the Euler gamma-function}\label{gamma} 
Denote $\Gamma(z)$ the Euler gamma-function.
Recall \cite[Section 1]{LY} and \cite[Theorem 2]{Ye}
\begin{equation}\label{Eq-13}
T(r,\Gamma)=\frac{1}{\pi}\,r\log r+O(r),
\end{equation}
\begin{equation}\label{Eq-14}
N(r,\Gamma)=r+O(\log r),
\end{equation}
\begin{equation}\label{Eq-15}
N\Big(r,\frac{1}{\Gamma-c}\Big)=\frac{1}{\pi}\,r\log r+O(r)\hspace{2mm}\mathrm{for}\hspace{2mm}c\neq0,\infty.
\end{equation}
This inspires us to introduce the following family of meromorphic functions in $\C$.

\begin{definition}\label{De1}
Let $\X_{p,q}$ be the family of meromorphic functions $\eta$ in $\C$ such that
\begin{equation}\label{Eq-16}
T(r,\eta)=K_1\,r^p\log^qr+O(r^p\log^{q-1}r),
\end{equation}
\begin{equation}\label{Eq-17}
\bar{N}\Big(r,\frac{1}{\eta-c}\Big)\geq K_2\,r^p\log^qr+O(r^p\log^{q-1}r).
\end{equation}
Here, $c\in\C$ is a finite value, $p,q\geq1$ are integers, and $K_1\geq K_2>0$ are real numbers.
\end{definition}

One observes $\zeta,\Gamma\in\X_{1,1}$.
In fact, by Conrey \cite[Theorem 1]{Co} and the Riemann-von Mangoldt formula \cite[Theorem 9.4]{Ti} (see the proof of \cite[Theorem 1.2]{GL}), it follows that
\begin{equation}\label{Eq-18}
\bar{N}\Big(r,\frac{1}{\zeta}\Big)\geq\frac{1}{6\hspace{0.2mm}\pi}\,r\log r+O(r),
\end{equation}
which together with \eqref{Eq-3} implies $\zeta\in\X_{1,1}$ for $c=0$, $K_1=\frac{1}{\pi}$ and $K_2=\frac{1}{6\pi}$; on the other hand, seeing that
\begin{equation}
\bar{N}\Big(r,\frac{1}{\Gamma-c}\Big)+N\Big(r,\frac{1}{\Gamma'}\Big)\geq N\Big(r,\frac{1}{\Gamma-c}\Big)+O(1)\nonumber
\end{equation}
and recalling that $0$ is the only Picard value of $\Gamma$ (yet both $0,\infty$ are Nevanlinna's defect values of $\Gamma$), one has
\begin{equation}
\begin{aligned}
&\,N\Big(r,\frac{1}{\Gamma'}\Big)=N\Big(r,\frac{1}{\Gamma'/\Gamma}\Big)\leq T\Big(r,\frac{\Gamma'}{\Gamma}\Big)+O(1)\\
\leq&\,N\Big(r,\frac{\Gamma'}{\Gamma}\Big)+m\Big(r,\frac{\Gamma'}{\Gamma}\Big)+O(1)\leq r+O(\log r)\nonumber
\end{aligned}
\end{equation}
in light of the lemma of logarithmic derivative and \eqref{Eq-14} so that
\begin{equation}\label{Eq-19}
\bar{N}\Big(r,\frac{1}{\Gamma-c}\Big)=\frac{1}{\pi}\,r\log r+O(r),
\end{equation}
which along with \eqref{Eq-13} yields $\Gamma\in\X_{1,1}$ for all $c\neq0,\infty$ and $K_1=K_2=\frac{1}{\pi}$ with equality.

Now, we can formulate our main result of this section as follows.

\begin{theorem}\label{Th3}
Assume $\eta\in\X_{p,q}$ is a meromorphic function associated with the numbers $c\in\C$, $p,q\geq1$, and $K_1\geq K_2>0$.
Let $\displaystyle{a\neq b\in\C\cup\bre{\infty}}$ be distinct from $c$, and let $f$ be a meromorphic function in $\C$ such that either $\displaystyle{\limsup\limits_{r\to+\infty}\frac{\bar{N}\big(r,\frac{1}{f-c}\big)}{T(r,f)}<1}$ or $\displaystyle{\bar{N}\Big(r,\frac{1}{f-c}\Big)=O(T(r,\eta))}$.
When $f,\eta$ share the values $a,b$ counting multiplicity, with $\displaystyle{\eta^{-1}(c)\subseteq f^{-1}(c)}$ ignoring multiplicity, except possibly at a set $\E$ of points with $\displaystyle{n(r,\E)=o(r^p\log^{q-1}r)}$, then $f=\eta$.
Here, $n(r,\E)$ denotes the counting function of $\E$, \textit{i.e.}, the number of points in the set $\displaystyle{\E\cap\bre{z\in\C:\n{z}<r}}$.
\end{theorem}

\begin{proof}
First, consider finite values $a,b$ and define an auxiliary function
\begin{equation}\label{Eq-20}
G:=\frac{\eta-a}{f-a}\cdot\frac{f-b}{\eta-b}.
\end{equation}
Suppose $\E_1$ and $\E_2$ are the zero and pole sets of $G$.
Then, it follows from our assumption that $\displaystyle{n(r,\E_1)=o(r^p\log^{q-1}r)}$ and $\displaystyle{n(r,\E_2)=o(r^p\log^{q-1}r)}$.

Below, we follow closely the method from Li \cite{Li2} (see also Han \cite{Ha} and L\"{u} \cite{Lu}), and assume that $\bre{a_k:k\geq1}$ and $\bre{b_k:k\geq1}$ are the non-zero elements of $\E_1$ and $\E_2$ arranged in ascending orders, respectively, \textit{i.e.}, $\n{a_k}\leq\n{a_{k+1}}$ and $\n{b_k}\leq\n{b_{k+1}}$.
Construct two infinite products
\begin{equation}
\Pi_1(z):=\prod_{k=1}^{+\infty}E\pr{\frac{z}{a_k},p}\hspace{2mm}\mathrm{and}\hspace{2mm}\Pi_2(z):=\prod_{k=1}^{+\infty}E\pr{\frac{z}{b_k},p},\nonumber
\end{equation}
where, for the integer $p\geq1$, we write $\displaystyle{E(z,p):=(1-z)\hspace{0.2mm}e^{z+\frac{1}{2}z^2+\cdots+\frac{1}{p}z^p}}$.
Then, $\Pi_1$ and $\Pi_2$ are entire functions defined in $\C$ having $a_k$'s and $b_k$'s as their zeros, respectively.

Actually, using the Stieltjes integral and seeing $\displaystyle{n(r,\E_1)=o(r^p\log^{q-1}r)}$, one has
\begin{equation}
\begin{aligned}
&\,\sum_{k=1}^{+\infty}\frac{1}{\n{a_k}^{p+1}}=\int_{\n{a_1}}^{+\infty}\frac{d\hspace{0.2mm}(n(t,\E_1))}{t^{p+1}}
=\lim_{r\to+\infty}\int_{\n{a_1}}^r\frac{d\hspace{0.2mm}(n(t,\E_1))}{t^{p+1}}\\
=&\,\lim_{r\to+\infty}\frac{n(r,\E_1)}{r^{p+1}}+O(1)+(p+1)\lim_{r\to+\infty}\int_{\n{a_1}}^r\frac{n(t,\E_1)}{t^{p+2}}\,dt\\
\leq&\,\lim_{r\to+\infty}\frac{\log^{q-1}r}{r}+(p+1)\lim_{r\to+\infty}\int_{r_0}^r\frac{2}{t^{1.2}}\,dt+O(1)<+\infty,\nonumber
\end{aligned}
\end{equation}
with $r_0\gg1$ satisfying $\displaystyle{\log^{q-1}r\leq2\hspace{0.2mm}r^{0.8}}$ for all $r\geq r_0$; this combined with Hayman \cite[Theorem 1.11]{Hay} and Goldberg-Ostrovskii \cite[p.56, Remark 1]{GO} shows that $\Pi_1(z)$ is an entire function in $\C$ having zeros $a_k$'s, which in addition satisfies for $\displaystyle{C_p:=2\hspace{0.2mm}(p+1)(2+\log p)>0}$,
\begin{equation}\label{Eq-21}
\log M(r,\Pi_1)\leq C_p\bre{r^p\int_{\n{a_1}}^r\frac{n(t,\E_1)}{t^{p+1}}\,dt+r^{p+1}\int_r^{+\infty}\frac{n(t,\E_1)}{t^{p+2}}\,dt}.
\end{equation}

Now, for each $\epsilon>0$, there exists an $r_1\gg1$ such that $\displaystyle{n(r,\E_1)\leq\epsilon\hspace{0.2mm}r^p\log^{q-1}r}$ for all $r\geq r_1$.
Combing this and \eqref{Eq-21} with \cite[p.37, Theorem 7.1]{GO}, it follows that
\begin{equation}\label{Eq-22}
\begin{aligned}
T(r,\Pi_1)&\leq\max\bre{\log M(r,\Pi_1),0}\leq C_p\bre{r^p\int_{\n{a_1}}^r\frac{n(t,\E_1)}{t^{p+1}}\,dt+r^{p+1}\int_r^{+\infty}\frac{n(t,\E_1)}{t^{p+2}}\,dt}\\
          &\leq C_p\bre{r^p\int_{\n{a_1}}^{r_1}\frac{n(t,\E_1)}{t^{p+1}}\,dt+r^p\int_{r_1}^r\frac{n(t,\E_1)}{t^{p+1}}\,dt+r^{p+1}\int_r^{+\infty}\frac{n(t,\E_1)}{t^{p+2}}\,dt}\\
          &\leq C_p\bre{r^p\int_{r_1}^r\frac{\epsilon\log^{q-1}t}{t}\,dt+r^{p+1}\int_r^{+\infty}\frac{\epsilon\log^{q-1}t}{t^2}\,dt+O(1)}\\
          &\leq C_p\bre{\frac{\epsilon}{q}\hspace{0.2mm}r^p\log^qr+\epsilon\hspace{0.2mm}r^p\log^{q-1}r+O(\epsilon\hspace{0.2mm}r^p\log^{q-2}r)}\\
          &\leq\frac{K_2}{3}\,r^p\log^qr+O(r^p\log^{q-1}r)
\end{aligned}
\end{equation}
for all $r\geq r_1$, with $\epsilon>0$ taken smaller if necessary.
Here, routine substitution and integration by parts were used.
Similarly, we have for some $r_2\gg1$ and all $r\geq r_2$,
\begin{equation}\label{Eq-23}
T(r,\Pi_2)\leq\frac{K_2}{3}\,r^p\log^qr+O(r^p\log^{q-1}r).
\end{equation}

Next, for the possible exceptional set $\E$ with $\displaystyle{n(r,\E)=o(r^p\log^{q-1}r)}$, one has
\begin{equation}\label{Eq-24}
N(r,\E)=\int_0^r\frac{n(t,\E)-n(0,\E)}{t}\,dt+n(0,\E)\log r=O(r^p\log^{q-1}r).
\end{equation}
Like \eqref{Eq-4}, using Nevanlinna's first and second main theorems, it yields that
\begin{equation}\label{Eq-25}
\begin{aligned}
T(r,f)&\leq\bar{N}\Big(r,\frac{1}{f-a}\Big)+\bar{N}\Big(r,\frac{1}{f-b}\Big)+\bar{N}\Big(r,\frac{1}{f-c}\Big)+S(r,f)\\
      &\leq2\hspace{0.2mm}T(r,\eta)+N(r,\E)+\bar{N}\Big(r,\frac{1}{f-c}\Big)+S(r,f)+O(1).
\end{aligned}
\end{equation}
Now, when $\displaystyle{\iota:=\limsup\limits_{r\to+\infty}\frac{\bar{N}\big(r,\frac{1}{f-c}\big)}{T(r,f)}<1}$ holds, one has from \eqref{Eq-24} and \eqref{Eq-25} that
\begin{equation}\label{Eq-26}
\pr{1-\frac{1+\iota}{2}-o(1)}T(r,f)\leq2\hspace{0.2mm}T(r,\eta)+O(r^p\log^{q-1}r);
\end{equation}
while when $\displaystyle{\bar{N}\Big(r,\frac{1}{f-c}\Big)=O(T(r,\eta))}$ holds, one sees from \eqref{Eq-24} and \eqref{Eq-25} that
\begin{equation}\label{Eq-27}
\pr{1-o(1)}T(r,f)\leq O(T(r,\eta))+O(r^p\log^{q-1}r).
\end{equation}
Therefore, by \eqref{Eq-16}, one arrives at $\displaystyle{T(r,f)\leq O(r^p\log^qr)+O(r^p\log^{q-1}r)}$ (for all $r$ outside of a possible set of finite Lebesgue measure), which gives $\rho(f)\leq p$ (so there is no exceptional set).
Thus, $\rho(G)\leq p$ and by virtue of Hadamard factorization theorem, it leads to
\begin{equation}
G=\frac{\eta-a}{f-a}\cdot\frac{f-b}{\eta-b}=z^l\hspace{0.2mm}e^P\hspace{0.2mm}\frac{\Pi_1}{\Pi_2}.\nonumber
\end{equation}
Here, $l$ is an integer that is the multiplicity of zero or pole of $G$ at $z=0$ and $P$ is a polynomial with $\deg P\leq\max\bre{\rho(\Pi_1),\rho(\Pi_2),\rho(G)}\leq p$ in view of \eqref{Eq-22} and \eqref{Eq-23}.

Observe that $G\equiv1$.
In fact, as $\displaystyle{\eta^{-1}(c)\subseteq f^{-1}(c)}$ except possibly at the set $\E$ of points with $\displaystyle{n(r,\E)=o(r^p\log^{q-1}r)}$, for the given $\epsilon>0$, there is an $\tilde{r}\geq\max\bre{r_1,r_2}$ such that
\begin{equation}\label{Eq-28}
\begin{aligned}
\bar{N}\Big(r,\frac{1}{\eta-c}\Big)&\leq N\Big(r,\frac{1}{G-1}\Big)+N(r,\E)\leq T(r,G)+N(r,\E)+O(1)\\
                                   &\leq T(r,\Pi_1)+T(r,\Pi_2)+T\big(r,z^l\big)+T\big(r,e^P\big)+N(r,\E)+O(1)\\
                                   &\leq \frac{2K_2}{3}\,r^p\log^qr+O(r^p\log^{q-1}r)
\end{aligned}
\end{equation}
provided $r\geq\tilde{r}$.
This contradicts our hypothesis \eqref{Eq-17}.
As a consequence, we must have $G\equiv1$, which further implies $f=\eta$.

Finally, if one of $a,b$ is $\infty$, assume without loss of generality $a$ is finite and $b=\infty$.
Consider $\displaystyle{\tilde{f}:=\frac{1}{f-d},\tilde{\eta}:=\frac{1}{\eta-d}}$ that share the finite values $\displaystyle{\tilde{a}:=\frac{1}{a-d},\tilde{b}:=0}$ counting multiplicity, plus the finite value $\displaystyle{\tilde{c}:=\frac{1}{c-d}}$ such that they satisfy all the given conditions, with $d\in\C$ finite and distinct from $a,c$.
Exactly the same analyses as above imply $\tilde{f}=\tilde{\eta}$. So, $f=\eta$.
\end{proof}

Note $0$ is the unique Picard value of $\Gamma(z)$ and $s=1$ is the unique, simple pole of $\zeta(s)$.
Using these nice properties of $\Gamma(z)$ and $\zeta(s)$, we consider a sub-family of $\X_{p,q}$ below.

\begin{definition}\label{De2}
Let $\displaystyle{\widetilde{\X}_{p,q}\subsetneq\X_{p,q}}$ be the sub-family of meromorphic functions $\eta$ in $\C$ such that $\displaystyle{n(r,\B_\eta)=o(r^p\log^{q-1}r)}$ for the $b$-points of each $\eta$ with some value $\displaystyle{b\in\C\cup\bre{\infty}\setminus\bre{c}}$.
\end{definition}

\begin{theorem}\label{Th4}
Suppose that $\eta\in\widetilde{\X}_{p,q}$ is a meromorphic function associated with the numbers $c\in\C$, $\displaystyle{b\in\C\cup\bre{\infty}\setminus\bre{c}}$, $p,q\geq1$, and $K_1\geq K_2>0$.
Choose $\displaystyle{a\in\C\cup\bre{\infty}\setminus\bre{b,c}}$, and let $f$ be a meromorphic function in $\C$ whose $b$-points satisfies $\displaystyle{n(r,\B_f)=o(r^p\log^{q-1}r)}$ as well such that either $\displaystyle{\limsup\limits_{r\to+\infty}\frac{\bar{N}\big(r,\frac{1}{f-c}\big)}{T(r,f)}<1}$ or $\displaystyle{\bar{N}\Big(r,\frac{1}{f-c}\Big)=O(T(r,\eta))}$.
When $f,\eta$ share $a$ counting multiplicity, with $\displaystyle{\eta^{-1}(c)\subseteq f^{-1}(c)}$ ignoring multiplicity, except possibly at a set $\E$ of points with $\displaystyle{n(r,\E)=o(r^p\log^{q-1}r)}$, then $f=\eta$.
\end{theorem}

\begin{proof}
Without loss of generality, assume $b=\infty$; otherwise, set $\displaystyle{\hat{f}:=\frac{1}{f-b},\hat{\eta}:=\frac{1}{\eta-b}}$ as well as $\displaystyle{\hat{a}:=\frac{1}{a-b},\hat{b}:=\infty,\hat{c}:=\frac{1}{c-b}}$.
Consider an auxiliary function
\begin{equation}\label{Eq-29}
G_1:=\frac{\eta-a}{f-a}.
\end{equation}

Denote by $\E_1$ and $\E_2$ the zero and pole sets of $G_1$.
Then, one sees $\displaystyle{\E_1\cup\E_2=\E\cup\E_\eta\cup\E_f}$ and $\displaystyle{n(r,\E_1),n(r,\E_2)=o(r^p\log^{q-1}r)}$.
Write the non-zero elements of $\E_1$ and $\E_2$ as $\bre{a_k:k\geq1}$ and $\bre{b_k:k\geq1}$ with $\n{a_k}\leq\n{a_{k+1}}$ and $\n{b_k}\leq\n{b_{k+1}}$, and construct two infinite products
\begin{equation}
\Pi_1(z):=\prod_{k=1}^{+\infty}E\pr{\frac{z}{a_k},p}\hspace{2mm}\mathrm{and}\hspace{2mm}\Pi_2(z):=\prod_{k=1}^{+\infty}E\pr{\frac{z}{b_k},p}.\nonumber
\end{equation}
Then, $\Pi_1$ and $\Pi_2$ are entire functions in $\C$ having $a_k$'s and $b_k$'s as their zeros, respectively, and \eqref{Eq-22} and \eqref{Eq-23} follow analogously.
Moreover, like \eqref{Eq-24}, one observes
\begin{equation}\label{Eq-30}
N(r,\E\cup\E_\eta\cup\E_f)=O(r^p\log^{q-1}r).
\end{equation}

All these modifications plus the proof of Theorem \ref{Th3} lead to
\begin{equation}
G_1=\frac{\eta-a}{f-a}=z^l\hspace{0.2mm}e^P\hspace{0.2mm}\frac{\Pi_1}{\Pi_2}.\nonumber
\end{equation}
Here, $l$ is an integer that is the multiplicity of zero or pole of $G_1$ at $z=0$ and $P$ is a polynomial with $\deg P\leq\max\bre{\rho(\Pi_1),\rho(\Pi_2),\rho(G_1)}\leq p$ in view of \eqref{Eq-22} and \eqref{Eq-23}.
Using the hypothesis $\displaystyle{\eta^{-1}(c)\subseteq f^{-1}(c)}$ and \eqref{Eq-30}, one similarly has $G_1\equiv1$ that further implies $f=\eta$.
\end{proof}

\begin{remark}\label{Re2}
Seeing \eqref{Eq-14}, the preceding Theorem \ref{Th4} in particular implies when a meromorphic function $f$ in $\C$ having only finitely many zeros shares two values $\displaystyle{a\in\C\cup\bre{\infty}\setminus\bre{0}}$ {\rm CM} and $\displaystyle{b\in\C\setminus\bre{a,0}}$ {\rm IM} with the Euler gamma-function $\Gamma$, then $f=\Gamma$.
On the other hand, the preceding Theorem \ref{Th2} in particular implies when a meromorphic function $f$ in $\C$ having only finitely many poles shares two values $\displaystyle{a\neq b\in\C}$ $a$ {\rm CM} and $b$ {\rm IM} with the Riemann zeta-function $\zeta$, then $f=\zeta$.
These results consist with Li \cite{Li1} and Han \cite[Theorem 2.1]{Ha}.
\end{remark}

\section{Results regarding the Riemann zeta-function II}\label{zeta2} 
Denote $\M$ the space of meromorphic functions in $\C$, and $\N$ the space of entire functions in $\C$; denote $\M_1$ the space of meromorphic functions in $\C$ that have finite non-integral order or have integral order yet are of maximal growth type, and $\N_1$ the space of such entire functions in $\C$.
Then, from the classical result of Nevanlinna \cite{Ne1}, and Han \cite{Ha}, one has

\vspace{1.8mm}
\begin{itemize}
  \item $\left\{\begin{array}{lll}
        f,g\in\M\ {\rm are}\ {\rm identical}\ {\rm under}\ {\rm``5\ IM"}\ {\rm value}\ {\rm sharing}\ {\rm condition}. \\ \\
        f\in\M_1\ {\rm and}\ g\in\M\ {\rm are}\ {\rm identical}\ {\rm under}\ {\rm``1\ CM + 3\ IM"}\ {\rm value}\ {\rm sharing}\ {\rm condition}. \\ \\
        f\in\M_1\ {\rm and}\ g\in\M\ {\rm are}\ {\rm identical}\ {\rm under}\ {\rm``3\ CM"}\ {\rm value}\ {\rm sharing}\ {\rm condition}.
        \end{array}\right.$\\ \\
  \item $\left\{\begin{array}{lll}
        f,g\in\N\ {\rm are}\ {\rm identical}\ {\rm under}\ {\rm``4\ IM"}\ {\rm value}\ {\rm sharing}\ {\rm condition}. \\ \\
        f\in\N_1\ {\rm and}\ g\in\N\ {\rm are}\ {\rm identical}\ {\rm under}\ {\rm``3\ IM"}\ {\rm value}\ {\rm sharing}\ {\rm condition}. \\ \\
        f\in\N_1\ {\rm and}\ g\in\N\ {\rm are}\ {\rm identical}\ {\rm under}\ {\rm``1\ CM + 1\ IM"}\ {\rm value}\ {\rm sharing}\ {\rm condition}.
        \end{array}\right.$
\end{itemize}
\vspace{1.8mm}

Below, we follow Hu and Li \cite{HL} to discuss the remanding case regarding the number ``$2$'' for meromorphic functions in $\C$ and the number ``$1$'' for entire functions in $\C$; an earlier result on meromorphic functions $f$ in $\C$ with $\rho(f)<1$ was given by Rao \cite[Theorem 4]{Ra}.
Here, we try to describe some more general results under possibly the minimum requirement.

\begin{theorem}\label{Th5}
Let $f,g$ be meromorphic functions in $\C$ of finite order, and let $\displaystyle{a,b,c\in\C\cup\bre{\infty}}$ be values pairwise distinct from each other.
When $f,g$ share the values $a,b$ counting multiplicity with $\displaystyle{\lim_{x\to+\infty}f(z)=\lim_{x\to+\infty}g(z)=c}$ uniformly in $y$ for $z=x+\Ic y\in\C$, then $f=g$.
\end{theorem}

\begin{proof}
Without loss of generality, assume $b=\infty$; otherwise, set $\displaystyle{\hat{f}:=\frac{1}{f-b},\hat{g}:=\frac{1}{g-b}}$ as well as $\displaystyle{\hat{a}:=\frac{1}{a-b},\hat{b}:=\infty,\hat{c}:=\frac{1}{c-b}}$.
Consider an auxiliary function
\begin{equation}\label{Eq-31}
L:=\frac{f-a}{g-a}.
\end{equation}
By virtue of our hypotheses, one finds a polynomial $P$ such that $L=e^P$.
Write
\begin{equation}\label{Eq-32}
\Re\hspace{0.2mm}(P(x+\Ic y))=a_m(y)\hspace{0.2mm}x^m+a_{m-1}(y)\hspace{0.2mm}x^{m-1}+\cdots+a_0(y),
\end{equation}
with $\displaystyle{a_m(y),a_{m-1}(y),\ldots,a_0(y)}$ real polynomials in $y$.
If $a_m(y)\not\equiv0$ and $m\geq1$, one may take some $y_0$ such that $a_m(y_0)\neq0$; then, a straightforward analysis leads to
\begin{equation}\label{Eq-33}
1=\lim_{x\to+\infty}\n{L(z)}=\lim_{x\to+\infty}e^{\Re\hspace{0.2mm}(P(x+\Ic y))}=\left\{\begin{array}{ll}
+\infty &\mathrm{when}\hspace{2mm}a_m(y_0)>0\\ \\
0 &\mathrm{when}\hspace{2mm}a_m(y_0)<0
\end{array}\right.
\end{equation}
that is absurd.
This shows $\Re\hspace{0.2mm}(P(x+\Ic y))=a_0(y)$, and hence $\n{L(z)}=e^{a_0(y)}$, independent of $x$.
Letting $x\to+\infty$ again, we have $e^{a_0(y)}=1$ for all $y$; that is, $\n{L}\equiv1$ or $L$ is simply a constant.
Letting $x\to+\infty$ again, one derives $L\equiv1$ so that $f=g$.
\end{proof}

\begin{corollary}\label{Co1}
Assume $f,g$ are finite order meromorphic functions in $\C$ having $b$ as a common Picard value, and take $\displaystyle{a\neq c\in\C\cup\bre{\infty}\setminus\bre{b}}$.
When $f,g$ share the value $a$ counting multiplicity with $\displaystyle{\lim_{x\to+\infty}f(z)=\lim_{x\to+\infty}g(z)=c}$ uniformly in $y$ for $z=x+\Ic y\in\C$, then $f=g$.
\end{corollary}

\begin{proof}
Without loss of generality, assume $b=\infty$ and consider the function $L(z)$ in \eqref{Eq-31}; then, parallel discussion as that for Theorem \ref{Th5} implies $f=g$.
\end{proof}

\begin{remark}\label{Re3}
Upon a slight modification, the preceding Corollary \ref{Co1} in particular applies for the Euler gamma-function $\Gamma(z)$ (as $0$ is the unique Picard value of $\Gamma(z)$ and $\displaystyle{\lim_{x\to+\infty}\Gamma^{-1}(z)=0}$ uniformly in $y$ for $z=x+\Ic y\in\C$) and the Riemann zeta-function $\zeta(s)$ (as $s=1$ is the unique, simple pole of $\zeta(s)$ and $\displaystyle{\lim_{\sigma\to+\infty}\zeta(s)=1}$ uniformly in $t$ for $s=\sigma+\Ic t\in\C$).
\end{remark}

Denote by $\M_2$ the space of meromorphic functions in $\C$ that are of finite order and satisfy $\displaystyle{\lim_{x\to+\infty}f(z)=c}$ with $c\in\C\cup\bre{\infty}$ uniformly in $y$ for $z=x+\Ic y\in\C$, and $\N_2$ the space of such entire functions in $\C$.
Then, we have just observed a general result as follows.

\begin{proposition}\label{Pr2}
$f,g\in\M_2$ are identical under {\rm``$2$ CM''} value sharing condition and $f,g\in\N_2$ are identical under {\rm``$1$ CM''} value sharing condition.
\end{proposition}

Finally, we shall focus on the Riemann zeta-function $\zeta(s)$ and reconsider Theorems \ref{Th1} and \ref{Th2}; our results will cover the case $c=0$ left open in Theorem 1.2 of \cite{Lu}.

\begin{theorem}\label{Th6}
Let $a,b,c\in\C$ be finite and distinct, and let $f$ be a meromorphic function in $\C$ of finite order.
When $f,\zeta$ share the finite values $a,b$ counting multiplicity, with $\displaystyle{\zeta^{-1}(c)\subseteq f^{-1}(c)}$ ignoring multiplicity, except possibly at finitely many points, then $f=\zeta$.
\end{theorem}

\begin{proof}
Consider the auxiliary function $F(z)$ in \eqref{Eq-1}; since $f,\zeta$ share $a,b$ CM except possibly at finitely many points and $f$ has finite order, there exists a polynomial $P$ and a rational function $R$ such that
\begin{equation}\label{Eq-34}
F=\frac{\zeta-a}{f-a}\cdot\frac{f-b}{\zeta-b}=R\hspace{0.2mm}e^P.
\end{equation}

We claim $\deg P=0$.
Otherwise, put $\displaystyle{P(s)=a_ms^m+a_{m-1}s^{m-1}+\cdots+a_0}$ with $m\cdot a_m\neq0$, where $a_j=\alpha_j+\Ic\tilde{\alpha}_j\in\C$ with real numbers $\alpha_j,\tilde{\alpha}_j$ for $j=0,1,\ldots,m$.
Then, one has
\begin{equation}\label{Eq-35}
\begin{aligned}
&\Re\hspace{0.2mm}\big(P\big(s=re^{\Ic\theta}\big)\big)=\bre{\alpha_m\cos(m\hspace{0.2mm}\theta)-\tilde{\alpha}_m\sin(m\hspace{0.2mm}\theta)}r^m\\
&\,+\bre{\alpha_{m-1}\cos((m-1)\theta)-\tilde{\alpha}_{m-1}\sin((m-1)\theta)}r^{m-1}+\cdots+\alpha_0.
\end{aligned}
\end{equation}

By Proposition \ref{Pr1}, $\zeta-c$ has infinitely many zeros $\displaystyle{\varpi_n=r_ne^{\Ic\theta_n}}$ in the vertical strip $Z_V$.
We assume without loss of generality $f(\varpi_n)=c$, since $\displaystyle{\zeta^{-1}(c)\subseteq f^{-1}(c)}$ except possibly at finitely many points.
Substitute $\varpi_n$ into \eqref{Eq-34} and take absolute value to observe
\begin{equation}\label{Eq-36}
\begin{aligned}
1\equiv&\,\n{F(\varpi_n)}=\n{R(\varpi_n)}e^{\Re\hspace{0.2mm}(P(\varpi_n))}
=\n{R(\varpi_n)}e^{\bre{\alpha_m\cos(m\hspace{0.2mm}\theta_n)-\tilde{\alpha}_m\sin(m\hspace{0.2mm}\theta_n)}r_n^m}\\
&\,\times e^{\bre{\alpha_{m-1}\cos((m-1)\theta_n)-\tilde{\alpha}_{m-1}\sin((m-1)\theta_n)}r_n^{m-1}+\cdots+\alpha_0}.
\end{aligned}
\end{equation}
Noticing the special form of $Z_V$, one may assume that $r_n\to+\infty$ and $\theta_n\to\frac{\pi}{2}$ when $n\to+\infty$.
So, \eqref{Eq-36} further leads to
\begin{equation}\label{Eq-37}
\alpha_m\cos\Big(m\hspace{0.2mm}\frac{\pi}{2}\Big)-\tilde{\alpha}_m\sin\Big(m\hspace{0.2mm}\frac{\pi}{2}\Big)=0.
\end{equation}
On the other hand, $\zeta-c$ has infinitely many zeros $\displaystyle{\omega_n=\tilde{r}_ne^{\Ic\tilde{\theta}_n}}$ in the horizontal strip $Z_H$.
We without loss of generality assume $f(\omega_n)=c$, so that one similarly derives
\begin{equation}\label{Eq-38}
\begin{aligned}
1\equiv&\,\n{F(\omega_n)}=\n{R(\omega_n)}e^{\Re\hspace{0.2mm}(P(\omega_n))}
=\n{R(\omega_n)}e^{\bre{\alpha_m\cos(m\hspace{0.2mm}\tilde{\theta}_n)-\tilde{\alpha}_m\sin(m\hspace{0.2mm}\tilde{\theta}_n)}\tilde{r}_n^m}\\
&\,\times e^{\bre{\alpha_{m-1}\cos((m-1)\tilde{\theta}_n)-\tilde{\alpha}_{m-1}\sin((m-1)\tilde{\theta}_n)}\tilde{r}_n^{m-1}+\cdots+\alpha_0}.
\end{aligned}
\end{equation}
Noticing the special form of $Z_H$, one may assume that $\tilde{r}_n\to+\infty$ and $\tilde{\theta}_n\to\pi$ when $n\to+\infty$.
So, \eqref{Eq-38} further leads to
\begin{equation}\label{Eq-39}
\alpha_m\cos(m\hspace{0.2mm}\pi)-\tilde{\alpha}_m\sin(m\hspace{0.2mm}\pi)=0.
\end{equation}
\eqref{Eq-37} and \eqref{Eq-39} combined yields that $\alpha_m=0$ and $m\geq2$ is an even integer.

The same argument as conducted above implies that for $\varpi_n\in Z_V$, one has
\begin{equation}\label{Eq-40}
\alpha_{m-1}\cos\Big((m-1)\frac{\pi}{2}\Big)-\tilde{\alpha}_{m-1}\sin\Big((m-1)\frac{\pi}{2}\Big)=0,
\end{equation}
while for $\omega_n\in Z_H$, one has
\begin{equation}\label{Eq-41}
\alpha_{m-1}\cos((m-1)\pi)-\tilde{\alpha}_{m-1}\sin((m-1)\pi)=0.
\end{equation}
\eqref{Eq-40} and \eqref{Eq-41} combined shows $\alpha_{m-1}=\tilde{\alpha}_{m-1}=0$ for the odd integer $m-1$.

Using induction for the analyses as described in \eqref{Eq-36}-\eqref{Eq-41} leads to
\begin{equation}\label{Eq-42}
\begin{aligned}
\alpha_j\cos\Big(j\hspace{0.2mm}\frac{\pi}{2}\Big)&-\tilde{\alpha}_j\sin\Big(j\hspace{0.2mm}\frac{\pi}{2}\Big)=0,\\ \\
\alpha_j\cos(j\hspace{0.2mm}\pi)&-\tilde{\alpha}_j\sin(j\hspace{0.2mm}\pi)=0,
\end{aligned}
\end{equation}
for $j=1,2,\ldots,m$.
As a result, it follows that $\alpha_j=0$ for $j=1,2,\ldots,m$ and $\tilde{\alpha}_1=\tilde{\alpha}_3=\cdots=\tilde{\alpha}_{m-3}=\tilde{\alpha}_{m-1}=0$.
Hence, we can rewrite $P(s)$ as
\begin{equation}
P(s)=\Ic\tilde{\alpha}_ms^m+\Ic\tilde{\alpha}_{m-2}s^{m-2}+\cdots+\Ic\tilde{\alpha}_2s^2+a_0.\nonumber
\end{equation}

Write $\varpi_n=\beta_n+\Ic\gamma_n\in Z_V$ for real $\beta_n,\gamma_n$.
Then, one may assume that $\beta_n\to\beta_0\in\br{\frac{1}{4},\frac{3}{4}}$ and $\gamma_n\to+\infty$ as $n\to+\infty$.
Thus, recalling $m\geq2$ is an even integer, one has
\begin{equation}\label{Eq-43}
\begin{aligned}
1\equiv&\,\n{F(\varpi_n)}=\n{R(\varpi_n)}e^{\Re\hspace{0.2mm}(P(\varpi_n))}\\
=&\,\n{R(\varpi_n)}e^{\Re\hspace{0.2mm}(\Ic\tilde{\alpha}_m(\beta_n+\Ic\gamma_n)^m+\Ic\tilde{\alpha}_{m-2}(\beta_n+\Ic\gamma_n)^{m-2}+\cdots+a_0)}\\
=&\,\n{R(\varpi_n)}e^{\pr{m(-1)^{\frac{m}{2}}\tilde{\alpha}_m\beta_n\gamma_n^{m-1}+O(\gamma_n^{m-2})}},
\end{aligned}
\end{equation}
which further implies $m(-1)^{\frac{m}{2}}\tilde{\alpha}_m\beta_0=0$ that is impossible because $m\cdot a_m\neq0$ yet $\alpha_m=0$ (so $\tilde{\alpha}_m\neq0$) and $\beta_0\in\br{\frac{1}{4},\frac{3}{4}}$.
This contradiction can be interpreted as saying that $a_m=a_{m-1}=\cdots=a_1=0$; that is, $P=a_0$ or $\deg P=0$.
As a consequence, we have
\begin{equation}\label{Eq-44}
F=\frac{\zeta-a}{f-a}\cdot\frac{f-b}{\zeta-b}=R.
\end{equation}
As $\displaystyle{\zeta^{-1}(c)\subseteq f^{-1}(c)}$ except possibly at finitely many points, and $\zeta-c=0$ has infinitely many roots, we must have $R(s)\equiv1$ noticing $R$ is rational.
It then follows that $f=\zeta$.
\end{proof}

\begin{corollary}\label{Co2}
Let $a\neq c\in\C$ be finite, and let $f$ be a finite order meromorphic function in $\C$ having only finitely many poles.
When $f,\zeta$ share $a$ counting multiplicity, with $\displaystyle{\zeta^{-1}(c)\subseteq f^{-1}(c)}$ ignoring multiplicity, except possibly at finitely many points, then $f=\zeta$.
\end{corollary}

\begin{proof}
Consider the auxiliary function $F_1(z)$ in \eqref{Eq-10}.
Then, one has
\begin{equation}\label{Eq-45}
F_1=\frac{\zeta-a}{f-a}=R\hspace{0.2mm}e^P.
\end{equation}
Here, $P$ is a polynomial and $R$ is a rational function.

Next, one can exploit exactly the same analyses as described in the proof of Theorem \ref{Th6} to deduce that $P$ is simply a constant.
We then have
\begin{equation}\label{Eq-46}
F_1=\frac{\zeta-a}{f-a}=R.
\end{equation}
As $\displaystyle{\zeta^{-1}(c)\subseteq f^{-1}(c)}$ except possibly at finitely many points, and $\zeta-c=0$ has infinitely many roots, we must have $R(s)\equiv1$ noticing $R$ is rational.
It then follows that $f=\zeta$.
\end{proof}

\end{document}